\documentclass[12pt, reqno]{amsart}

\usepackage{amsmath,amsthm,amscd,amsfonts,amssymb}
\usepackage[bookmarksnumbered, plainpages]{hyperref}

\usepackage{url}

\newtheorem{theorem}{Theorem}[section]

\newtheorem{corollary}[theorem]{Corollary}
\theoremstyle{definition}

\theoremstyle{remark}
\newtheorem{remark}[theorem]{Remark}
\numberwithin{equation}{section}

\begin{document}

\title[Two Dimensional Diamond-alpha Inequalities on Time
Scales]{H\"{o}lder's and Hardy's Two Dimensional Diamond-alpha Inequalities on Time Scales\footnote{Accepted for publication (October 21, 2009) in the journal \emph{Annals of the University
of Craiova, Mathematics and Computer Science Series}
(\url{http://inf.ucv.ro/~ami}).}}

\author[M. R. Sidi Ammi, D. F. M. Torres]{Moulay Rchid Sidi
Ammi$^1$, Delfim F. M. Torres$^2$}

\address{$^{1}$ Department of Mathematics,
Fac. des Sci. et Tech. (FST-Errachidia),
University My Ismail, BP: 509 Boutalamine,
Errachidia, 52000, Morocco.}

\email{sidiammi@ua.pt}

\address{$^{2}$ African Institute for Mathematical Sciences,
6-8 Melrose Road, Muizenberg 7945, Cape Town, South Africa;
\newline
On leave of absence from:\newline
Department of Mathematics,
University of Aveiro, 3810-193 Aveiro, Portugal.}
\email{delfim@aims.ac.za, delfim@ua.pt}



\date{}


\begin{abstract}
We prove a two dimensional H\"{o}lder and reverse-H\"{o}lder
inequality on time scales via the diamond-alpha integral. Other
integral inequalities are established as well, which have as
corollaries some recent proved Hardy-type inequalities on time
scales.
\end{abstract}

\maketitle


\section{Introduction}

The theory and applications of dynamic derivatives on time scales
is receiving an increase of interest and attention. This relative
new area was created in order to unify and generalize discrete and
continuous analysis. It was introduced by Stefan Hilger
\cite{h2,h3}, then used as a tool in several computational and
numerical applications \cite{abra,b1,b2}. One important and very
active subject being developed within the theory of time scales is
the study of inequalities
\cite{abp,ozkan,Stef,srd,sidel,adnan,wong}.
The primary purpose of
this paper is to prove more general two dimensional
reverse-H\"{o}lder's and H\"{o}lder's inequalities on time scales,
using the recent theory of combined dynamic derivatives and the
more general notion of diamond-$\alpha$ integral
\cite{Rogers,Sheng,sfhd}. As particular cases, we get Hardy's
inequalities \cite{krm,adnan}.

H\"{o}lder's inequalities and their extensions have received
considerable attention in the theory
of differential and difference equations,
as well as other areas of mathematics
\cite{ozkan,srd,adnan,wong}. Recently, authors in
\cite{adnan} proved a time scale version of H\"{o}lder's
inequality in the two dimensional case, by using the
$\Delta$-integral. Here we extend this result to more general
diamond-$\alpha$ integral inequalities. The results in
\cite{adnan} are obtained choosing $\alpha = 1$; different
inequalities on time scales follow by choosing $0 \le \alpha < 1$
(\textrm{e.g.}, for $\alpha = 0$ one gets new
$\nabla$-integral inequalities).


\section{Preliminaries}
\label{sec:Prel}

A time scale $\mathbb{T}$ is an arbitrary nonempty closed subset
of the real numbers. Let $\mathbb{T}$ be a time scale with the
topology that it inherits from the real numbers. For $t \in
\mathbb{T}$, we define the forward jump operator $\sigma:
\mathbb{T} \rightarrow \mathbb{T}$ by $\sigma(t)= \inf \left \{s
\in \mathbb{T}: s >t \right\}$, and the backward jump operator
$\rho: \mathbb{T} \rightarrow \mathbb{T}$ by $\rho(t)= \sup \left
\{s \in \mathbb{T}: s < t \right \}$.

If $\sigma(t) > t$ we say that $t$ is right-scattered, while if
$\rho(t) < t$ we say that $t$ is left-scattered. Points that are
simultaneously right-scattered and left-scattered are said to be
isolated. If $\sigma(t)=t$, then $t$ is called right-dense; if
$\rho(t)=t$, then $t$ is called left-dense. Points that are
right-dense and left-dense at the same time are called dense. The
mappings $\mu, \nu: \mathbb{T} \rightarrow [0, +\infty)$ defined
by $\mu(t):=\sigma(t)-t$ and $\nu(t):=t-\rho(t)$ are called,
respectively, the forward and backward graininess function.

Given a time scale $\mathbb{T}$, we introduce the sets
$\mathbb{T}^{\kappa}$, $\mathbb{T}_{\kappa}$, and
$\mathbb{T}^{\kappa}_{\kappa}$ as follows. If $\mathbb{T}$ has a
left-scattered maximum $t_{1}$, then $\mathbb{T}^{\kappa}=
\mathbb{T}-\{t_{1} \}$, otherwise $\mathbb{T}^{\kappa}=
\mathbb{T}$. If  $\mathbb{T}$ has a right-scattered minimum
$t_{2}$, then $\mathbb{T}_{\kappa}= \mathbb{T}-\{t_{2} \}$,
otherwise $\mathbb{T}_{\kappa}= \mathbb{T}$. Finally,
$\mathbb{T}_{\kappa}^{\kappa}= \mathbb{T}^{\kappa} \bigcap
\mathbb{T}_{\kappa}$.

Let $f: \mathbb{T}\rightarrow \mathbb{R}$ be a real valued
function on a time scale $\mathbb{T}$. Then, for $t \in
\mathbb{T}^{\kappa}$, we define $f^{\Delta}(t)$ to be the number,
if one exists, such that for all $\epsilon >0$, there is a
neighborhood $U$ of $t$ such that for all $s \in U$,
$$
\left|f(\sigma(t))- f(s)-f^{\Delta}(t)(\sigma(t)-s)\right| \leq
\epsilon |\sigma(t)-s|.
$$
We say that $f$ is delta differentiable on $\mathbb{T}^{\kappa}$
provided $f^{\Delta}(t)$ exists for all $t \in
\mathbb{T}^{\kappa}$. Similarly, for $t \in \mathbb{T}_{\kappa}$
we define $f^{\nabla}(t)$ to be the number, if one exists, such
that for all $\epsilon >0$, there is a neighborhood $V$ of $t$
such that for all $s \in V$
$$
\left|f(\rho(t))- f(s)-f^{\nabla}(t)(\rho(t)-s)\right| \leq
\epsilon |\rho(t)-s|.
$$
We say that $f$ is nabla differentiable on $\mathbb{T}_{\kappa}$,
provided that $f^{\nabla}(t)$ exists for all $t \in
\mathbb{T}_{\kappa}$.

For $f:\mathbb{T}\rightarrow \mathbb{R}$ we define the function
$f^{\sigma}: \mathbb{T}\rightarrow \mathbb{R}$ by
$f^{\sigma}(t)=f(\sigma(t))$ for all $t \in \mathbb{T}$, that is,
$f^{\sigma}= f\circ \sigma$. Similarly, we define the function
$f^{\rho}: \mathbb{T}\rightarrow \mathbb{R}$ by
$f^{\rho}(t)=f(\rho(t))$ for all $t \in \mathbb{T}$, that is,
$f^{\rho}= f\circ \rho$.

A function $f: \mathbb{T} \rightarrow \mathbb{R} $ is called
rd-continuous, provided it is continuous at all right-dense points
in $\mathbb{T}$ and its left-sided limits finite at all left-dense
points in $\mathbb{T}$. A function $f: \mathbb{T} \rightarrow
\mathbb{R} $ is called ld-continuous, provided it is continuous at
all left-dense points in $\mathbb{T}$ and its right-sided limits
finite at all right-dense points in $\mathbb{T}$.

A function $F: \mathbb{T} \rightarrow \mathbb{R} $ is called a
delta antiderivative of $f: \mathbb{T} \rightarrow \mathbb{R}$,
provided $F^{\Delta}(t)=f(t)$ holds for all $t \in
\mathbb{T}^{\kappa}$. Then the delta integral of $f$ is defined by
$$\int^b_a f(t)\Delta t=F(b)-F(a) \, .$$

A function $G: \mathbb{T} \rightarrow \mathbb{R} $ is called a
nabla antiderivative of $g: \mathbb{T} \rightarrow \mathbb{R}$,
provided $G^{\nabla}(t)=g(t)$ holds for all $t \in
\mathbb{T}_{\kappa}$. Then the nabla integral of $g$ is defined by
$\int^b_a g(t)\nabla t=G(b)-G(a)$. For more on the delta and nabla
calculus on time scales, we refer the reader to \cite{abra,b1,b2}.
We review now the recent diamond-$\alpha$ derivative and
integral \cite{Rogers,Sheng,sfhd}.

Let $\mathbb{T}$ be a time scale and $f$ differentiable on
$\mathbb{T}$ in the $\Delta$ and $\nabla$ senses. For $t \in
\mathbb{T}$, we define the diamond-$\alpha$ dynamic derivative
$f^{\diamondsuit_{\alpha}}(t)$ by
$$
f^{\diamondsuit_{\alpha}}(t)= \alpha
f^{\Delta}(t)+(1-\alpha)f^{\nabla}(t), \quad 0 \leq \alpha \leq 1.
$$
Thus, $f$ is diamond-$\alpha$ differentiable if and only if $f$ is
$\Delta$ and $\nabla$ differentiable. The diamond-$\alpha$
derivative reduces to the standard  $\Delta$ derivative for
$\alpha =1$, or the standard $\nabla$ derivative for $\alpha =0$.
On the other hand, it represents a ``weighted derivative'' for
$\alpha \in (0,1)$. Diamond-$\alpha$ derivatives have shown in
computational experiments to provide efficient and balanced
approximation formulas, leading to the design of more reliable
numerical methods \cite{Sheng,sfhd}.

Let $f, g: \mathbb{T} \rightarrow \mathbb{R}$ be diamond-$\alpha$
differentiable at $t \in \mathbb{T}$. Then,

\begin{itemize}

\item[(i)] $f+g: \mathbb{T} \rightarrow \mathbb{R}$ is
diamond-$\alpha$ differentiable at $t \in \mathbb{T}$ with
$$ (f+g)^{\diamondsuit^{\alpha}}(t)=
(f)^{\diamondsuit^{\alpha}}(t)+(g)^{\diamondsuit^{\alpha}}(t).
$$

\item[(ii)] For any constant $c$, $cf:  \mathbb{T} \rightarrow
\mathbb{R}$
 is diamond-$\alpha$ differentiable at $t \in \mathbb{T}$ with
$$
(cf)^{\diamondsuit^{\alpha}}(t)= c(f)^{\diamondsuit^{\alpha}}(t).
$$

\item[(ii)]  $fg: \mathbb{T} \rightarrow \mathbb{R}$ is
diamond-$\alpha$ differentiable at $t \in \mathbb{T}$ with
$$
(fg)^{\diamondsuit^{\alpha}}(t)=
(f)^{\diamondsuit^{\alpha}}(t)g(t)+ \alpha
f^{\sigma}(t)(g)^{\Delta}(t) +(1-\alpha)
f^{\rho}(t)(g)^{\nabla}(t).
$$

\end{itemize}

Let $a, t \in  \mathbb{T}$, and $h: \mathbb{T} \rightarrow
\mathbb{R}$. Then, the diamond-$\alpha$ integral from $a$ to $t$
of $h$ is defined by
$$
\int_{a}^{t}h(\tau) \diamondsuit_{\alpha} \tau = \alpha
\int_{a}^{t}h(\tau) \Delta \tau +(1- \alpha) \int_{a}^{t}h(\tau)
\nabla \tau, \quad 0 \leq \alpha \leq 1 \, ,
$$
provided that there exist delta and nabla integrals of $h$ on
$\mathbb{T}$. It is clear that the diamond-$\alpha$ integral of
$h$ exists when $h$ is a continuous function. Let $a$, $b$, $t \in
\mathbb{T}$, $c \in \mathbb{R}$, and $f$ and $g$ be continuous
functions on $[a,b] \cap \mathbb{T}$. Then (\textrm{cf.}
\cite[Theorem~3.7]{sfhd} and \cite[Lemma~2.2]{srd}), the following
properties hold:

\begin{itemize}

\item[(a)] $\int_{a}^{t}\left( f(\tau)+g(\tau) \right)
\diamondsuit_{\alpha} \tau = \int_{a}^{t} f(\tau)
\diamondsuit_{\alpha} \tau + \int_{a}^{t} g(\tau)
\diamondsuit_{\alpha} \tau$;

\item[(b)] $\int_{a}^{t} c f(\tau) \diamondsuit_{\alpha} \tau = c
\int_{a}^{t} f(\tau) \diamondsuit_{\alpha} \tau$;

\item[(c)] $\int_{a}^{t} f(\tau) \diamondsuit_{\alpha} \tau =
\int_{a}^{b} f(\tau) \diamondsuit_{\alpha} \tau + \int_{b}^{t}
f(\tau) \diamondsuit_{\alpha} \tau$.

\item[(d)] If $f(t)\geq 0$ for all $t\in[a,b]_{\mathbb{T}}$,
    then $\int_a^b f(t)\Diamond_\alpha t\geq 0$.

\item[(e)] If $f(t)\leq g(t)$ for all $t\in[a,b]_{\mathbb{T}}$,
    then $\int_a^b f(t)\Diamond_\alpha t\leq\int_a^b g(t)\Diamond_\alpha t$.

\item[(f)] If $f(t)\geq 0$ for all $t\in[a,b]_{\mathbb{T}}$, then
$f(t)=0$ if and
    only if $\int_a^b f(t)\Diamond_\alpha t=0$.

\end{itemize}


\section{Main Results}
\label{sec:MR}

We prove new diamond-$\alpha$ inequalities. As
particular cases we get $\Delta$-inequalities on time scales for
$\alpha = 1$, and $\nabla$-inequalities on time scales when
$\alpha = 0$. In the sequel we use $[a, b]$ to denote
$[a, b] \cap \mathbb{T}$. We also suppose that all integrals converge.


\begin{theorem}[reverse diamond-$\alpha$ H\"{o}lder's inequality]
\label{thm1} Let $\mathbb{T}$ be a time scale, $a$, $b \in
\mathbb{T}$ with $a < b$, and $f$ and $g$ be two positive
functions satisfying $0< m \leq \frac{f^{p}}{g^{q}} \leq M <
+\infty $ on the set $[a, b]$. If $\frac{1}{p}+ \frac{1}{q}=1$
with $p> 1$, then
\begin{equation}
\label{eq:rdaHi} \left(\int_{a}^{b} f^{p}(t) \diamondsuit_{\alpha}
t \right)^{\frac{1}{p}} \left (  \int_{a}^{b} g^{q}(t)
\diamondsuit_{\alpha} t \right)^{\frac{1}{q}} \leq
\left(\frac{M}{m}\right)^{\frac{1}{pq}} \int_{a}^{b} f(t) g(t)
\diamondsuit_{\alpha} t.
\end{equation}
\end{theorem}

\begin{proof}
We have $\frac{f^{p}}{g^{q}} \leq M$. Then, $ f^{\frac{p}{q}} \leq M^{\frac{1}{q}} g$. Multiplying by $f>0$,  it follows that
$$
 f^{p}=  f^{1+\frac{p}{q}}\leq M^{\frac{1}{q}}fg.
$$
Using properties $(e)$ and $(b)$, we can write that
\begin{equation} \label{eq1}
\left (  \int_{a}^{b} f^{p}(t) \diamondsuit_{\alpha} t
\right)^{\frac{1}{p}} \leq M^{\frac{1}{pq}} \left ( \int_{a}^{b}
f(t) g(t) \diamondsuit_{\alpha} t \right)^{\frac{1}{p}}.
\end{equation}
In the same manner, we  have $m^{\frac{1}{p}} g^{\frac{q}{p}} \leq
f$. Then,
$$
\int_{a}^{b} m^{\frac{1}{p}} g^{q}(t) \diamondsuit_{\alpha} t =
m^{\frac{1}{p}} \int_{a}^{b}  g^{1+\frac{q}{p}}(t)
\diamondsuit_{\alpha} t \leq \int_{a}^{b} f(t) g(t)
\diamondsuit_{\alpha} t.
$$
We obtain that
\begin{equation} \label{eq2}
m^{\frac{1}{pq}} \left(\int_{a}^{b} g^{q}(t) \diamondsuit_{\alpha}
t \right)^{\frac{1}{q}} \leq \left (  \int_{a}^{b} f(t) g(t)
\diamondsuit_{\alpha} t \right)^{\frac{1}{q}}.
\end{equation}
Gathering \eqref{eq1} and \eqref{eq2}, the intended
inequality \eqref{eq:rdaHi} is proved.
\end{proof}


\begin{remark}
For the particular case $\mathbb{T}=\mathbb{R}$,
Theorem~\ref{thm1} gives \cite[Theorem~2.1]{krm}. For $\alpha =
1$, Theorem~\ref{thm1} coincides with \cite[Lemma~1]{adnan}.
\end{remark}


We now define the diamond-$\alpha$ integral for a function of two variables.
The double integral is defined as an iterated integral.
Let $\mathbb{T}$ be a time scale with $a, b \in \mathbb{T}$,
$a < b$, and $f$ be a real-valued function on $\mathbb{T} \times \mathbb{T}$.
Because we need notation for partial derivatives with respect to time scale
variables $x$ and $y$ we denote the time scale partial derivative of $f(x,y)$
with respect to $x$ by $f^{\diamondsuit_{\alpha}^1}(x,y)$ and let
$f^{\diamondsuit_{\alpha}^2}(x,y)$ denote the time scale partial derivative
with respect to $y$. Definition of these partial derivatives are now given.
Fix an arbitrary $y \in \mathbb{T}$. Then the diamond-$\alpha$ derivative
of function
\begin{gather*}
\mathbb{T} \rightarrow \mathbb{R} \\
x \mapsto f(x,y)
\end{gather*}
is denoted by $f^{\diamondsuit_{\alpha}^1}$.
Let now $x \in \mathbb{T}$. The diamond-$\alpha$ derivative
of function
\begin{gather*}
\mathbb{T} \rightarrow \mathbb{R} \\
y \mapsto f(x,y)
\end{gather*}
is denoted by $f^{\diamondsuit_{\alpha}^2}$.
If function $f$ has a $\diamondsuit_{\alpha}^1$ antiderivative $A$, \textrm{i.e.},
$A^{\diamondsuit_{\alpha}^1} = f$, and $A$ has a $\diamondsuit_{\alpha}^2$
antiderivative $B$, \textrm{i.e.}, $B^{\diamondsuit_{\alpha}^2} = A$, then
\begin{equation*}
\begin{split}
\int_{a}^{b} \int_{a}^{b} f(x, y) \diamondsuit_{\alpha} x \diamondsuit_{\alpha} y
&:= \int_{a}^{b} \left( A(b,y) - A(a,y)\right) \diamondsuit_{\alpha} y \\
&= B(b,b) - B(b,a) - B(a,b) + B(a,a) \, .
\end{split}
\end{equation*}
Note that $\left(B^{\diamondsuit_{\alpha}^2}\right)^{\diamondsuit_{\alpha}^1} = f$.

\begin{theorem}[two dimensional diamond-$\alpha$ H\"{o}lder's inequality]
\label{thm:2} Let $\mathbb{T}$ be a time scale, $a$, $b \in
\mathbb{T}$ with $a < b$, $f, g, h : [a, b] \times [a, b]
\rightarrow \mathbb{R}$ be $\diamondsuit_{\alpha}$-integrable
functions, and $\frac{1}{p}+ \frac{1}{q}=1$ with $ p> 1$. Then,
\begin{multline}\label{eq3}
\int_{a}^{b} \int_{a}^{b}|h(x, y)f(x, y)g(x, y)|
\diamondsuit_{\alpha} x \diamondsuit_{\alpha} y \\
\leq  \left ( \int_{a}^{b} \int_{a}^{b}|h(x, y)|f(x, y)|^{p}
\diamondsuit_{\alpha} x \diamondsuit_{\alpha} y
\right)^{\frac{1}{p}} \left ( \int_{a}^{b} \int_{a}^{b}|h(x,
y)|g(x, y)|^{q} \diamondsuit_{\alpha} x \diamondsuit_{\alpha} y
\right)^{\frac{1}{q}}.
\end{multline}
\end{theorem}

\begin{proof}
Inequality \eqref{eq3} is trivially true in the case when $f$ or
$g$ or $h$ is identically zero. Suppose that
$$
\left (\int_{a}^{b} \int_{a}^{b}|h(x, y)||f(x,
y)|^{\frac{1}{p}}\diamondsuit_{\alpha} x \diamondsuit_{\alpha} y
\right ) \left ( \int_{a}^{b} \int_{a}^{b}|h(x, y)||g(x,
y)|^{\frac{1}{q}} \diamondsuit_{\alpha} x \diamondsuit_{\alpha} y
\right)\neq 0 \, ,
$$
and let
$$
A(x, y)= \frac{|h^{\frac{1}{p}}(x, y)||f(x, y)|}{\int_{a}^{b}
\int_{a}^{b}|h(x, y)||f(x, y)|^{\frac{1}{p}} \diamondsuit_{\alpha}
x \diamondsuit_{\alpha} y} \, ,
$$
and
$$
B(x, y)= \frac{|h^{\frac{1}{q}}(x, y)||g(x, y)|}{\int_{a}^{b}
\int_{a}^{b}|h(x, y)||g(x, y)|^{\frac{1}{q}} \diamondsuit_{\alpha}
x \diamondsuit_{\alpha} y} \, .
$$
From the well-known Young's inequality $\xi \lambda \leq
\frac{\xi^{p}}{p}+ \frac{\lambda^{q}}{q}$, valid for nonnegative
real numbers $\xi$ and $\lambda$, we have that
\begin{equation*}
\begin{split}
\int_{a}^{b} & \int_{a}^{b} A(x, y) B(x, y) \diamondsuit_{\alpha}
x
\diamondsuit_{\alpha} y \\
& \leq  \int_{a}^{b} \int_{a}^{b} \left [ \frac{A^{p}(x, y)}{p}+
\frac{B^{q}(x, y)}{q}\right ]
\diamondsuit_{\alpha} x \diamondsuit_{\alpha} y\\
& \leq \frac{1}{p}\int_{a}^{b} \int_{a}^{b}
\frac{|h||f|^{p}\diamondsuit_{\alpha} x \diamondsuit_{\alpha}
y}{\int_{a}^{b} \int_{a}^{b}|h||f|^{p}} + \frac{1}{q}\int_{a}^{b}
\int_{a}^{b} \frac{|h||g|^{q}\diamondsuit_{\alpha} x
\diamondsuit_{\alpha} y}{\int_{a}^{b} \int_{a}^{b}|h||g|^{q}}\\
& \leq \frac{1}{p}+ \frac{1}{q}=1 \, ,
\end{split}
\end{equation*}
and the desired result follows.
\end{proof}

\begin{remark}
For the particular case $\alpha=1$, Theorem~\ref{thm:2} coincides
with \cite[Theorem~4]{adnan}.
\end{remark}

\begin{theorem}[two dimensional diamond-$\alpha$ Cauchy-Schwartz's inequality]
Let $\mathbb{T}$ be a time scale, $a$, $b \in \mathbb{T}$ with $a
< b$. For $\diamondsuit_{\alpha}$-integrable functions $f, g, h:
[a, b] \times [a, b] \rightarrow \mathbb{R}$, we have:
\begin{multline}\label{eq4}
\int_{a}^{b} \int_{a}^{b}|h(x, y)f(x, y)g(x, y)|
\diamondsuit_{\alpha} x \diamondsuit_{\alpha} y \\
\leq \sqrt{ \left ( \int_{a}^{b} \int_{a}^{b}|h(x, y)||f(x,
y)|^{2} \diamondsuit_{\alpha} x \diamondsuit_{\alpha} y \right)
 \left (
\int_{a}^{b} \int_{a}^{b}|h(x, y)||g(x, y)|^{2}
\diamondsuit_{\alpha} x \diamondsuit_{\alpha} y \right)}.
\end{multline}
\end{theorem}

\begin{proof}
The Cauchy-Schwartz inequality \eqref{eq4} is the particular case
$p=q=2$ of \eqref{eq3}.
\end{proof}

We now obtain some general results for estimating the diamond-alpha double integral
$\int_{a}^{b} \int_{a}^{b} K(x, y) f(x)g(y) \diamondsuit_{\alpha}
x \diamondsuit_{\alpha} y $.

\begin{theorem}[diamond-$\alpha$ Hardy-type inequalities]
\label{thm:da:Hineq} Let $\mathbb{T}$ be a time scale, $a$, $b \in
\mathbb{T}$ with $a < b$, and $K(x, y)$, $f(x)$, $g(y)$,
$\varphi(x)$, and $\psi(y)$ be nonnegative functions. Let $$F(x)=\int_{a}^{b} K(x, y)
\psi^{-p}(y)\diamondsuit_{\alpha}y$$ and $$G(y)=\int_{a}^{b} K(x,y) \varphi^{-q}(x) \diamondsuit_{\alpha} x \, ,$$
where
$\frac{1}{p}+\frac{1}{q}=1$, $p> 1$. Then, the two inequalities
\begin{multline}\label{eq5}
\int_{a}^{b} \int_{a}^{b} K(x, y) f(x)g(y) \diamondsuit_{\alpha} x
\diamondsuit_{\alpha} y \\
\leq \left (\int_{a}^{b} \varphi^{p}(x) F(x) f^{p}(x)
\diamondsuit_{\alpha} x \right)^{\frac{1}{p}}   \left (
\int_{a}^{b} \psi^{q}(y) G(y) g^{q}(y) \diamondsuit_{\alpha} y
\right)^{\frac{1}{p}}
\end{multline}
and
\begin{equation}\label{eq6}
\int_{a}^{b} G^{1-p}(y) \psi^{-p}(y) \left ( \int_{a}^{b} K(x, y)
f(x) \diamondsuit_{\alpha} x \right)^{p} \diamondsuit_{\alpha} y
\leq \int_{a}^{b} \varphi^{p}(x) F(x)
f^{p}(x)\diamondsuit_{\alpha} x
\end{equation}
hold and are equivalent.
\end{theorem}
Equation \eqref{eq6} is the
diamond-$\alpha$ Hardy's inequality.
\begin{proof}
First, we prove that \eqref{eq5} hold. Write
\begin{equation*}
\int_{a}^{b} \int_{a}^{b} K(x, y) f(x)g(y) \diamondsuit_{\alpha} x
\diamondsuit_{\alpha} y  = \int_{a}^{b} \int_{a}^{b} K(x, y)
f(x)\frac{\varphi(x)}{\psi(y)}g(y) \frac{\psi(y)}{\varphi(x)}
\diamondsuit_{\alpha} x \diamondsuit_{\alpha} y.
\end{equation*}
Applying H\"{o}lder's inequality on time scale, we have
\begin{multline*}
\int_{a}^{b} \int_{a}^{b} K(x, y) f(x)g(y) \diamondsuit_{\alpha} x
\diamondsuit_{\alpha} y \\
 \leq \left (\int_{a}^{b} \varphi^{p}(x) F(x) f^{p}(x)
\diamondsuit_{\alpha} x \right)^{\frac{1}{p}}   \left (
\int_{a}^{b} \psi^{q}(y) G(y) g^{q}(y) \diamondsuit_{\alpha} y
\right)^{\frac{1}{p}}.
\end{multline*}
Now we show that \eqref{eq5} is equivalent to \eqref{eq6}.
Suppose that inequality \eqref{eq5} is verified. Set
$$
g(y)= G^{1-p}(y) \psi^{-p}(y) \left ( \int_{a}^{b} K(x, y) f(x)
\diamondsuit_{\alpha} x \right)^{p-1} \, .
$$
Using \eqref{eq5} and the fact that $\frac{1}{p}+\frac{1}{q}=1$,
we obtain:
\begin{equation*}
\begin{split}
&\int_{a}^{b} G^{1-p}(y) \psi^{-p}(y) \left ( \int_{a}^{b} K(x, y)
f(x) \diamondsuit_{\alpha} x \right)^{p} \diamondsuit_{\alpha} y
\\
&= \int_{a}^{b} \int_{a}^{b} K(x, y) f(x) g(y)
\diamondsuit_{\alpha} x \diamondsuit_{\alpha} y \\
& \leq \left( \int_{a}^{b} \varphi^{p}(x) F(x) f^{p}(x)
\diamondsuit_{\alpha} x
 \right)^{\frac{1}{p}} \left( \int_{a}^{b} \psi^{q}(y) G(y) g^{q}(y) \diamondsuit_{\alpha}
 y \right)^{\frac{1}{q}} \\
 & = \left( \int_{a}^{b} \varphi^{p}(x) F(x) f^{p}(x)
\diamondsuit_{\alpha} x
 \right)^{\frac{1}{p}} \\
& \qquad \qquad \cdot \left ( \int_{a}^{b} G^{1-p}(y) \psi^{-p}(y) \left( \int_{a}^{b} K(x, y) f(x)
 \diamondsuit_{\alpha} x\right)^{p}   \diamondsuit_{\alpha} y \right
 )^{\frac{1}{q}}.
\end{split}
\end{equation*}
Inequality \eqref{eq6} is obtained by dividing both sides of the
previous inequality by
$$
\left ( \int_{a}^{b} G^{1-p}(y) \psi^{-p}(y) \left( \int_{a}^{b}
K(x, y) f(x)
 \diamondsuit_{\alpha} x\right)^{p}   \diamondsuit_{\alpha} y \right
 )^{\frac{1}{q}}.
 $$
Reciprocally, suppose that \eqref{eq6} is valid. From
H\"{o}lder's inequality we can write that
\begin{equation*}
\begin{split}
\int_{a}^{b} & \int_{a}^{b} K(x, y) f(x) g(y)
\diamondsuit_{\alpha} x \diamondsuit_{\alpha} y \\
&= \int_{a}^{b} \left ( \psi^{-1}(y) G^{\frac{-1}{q}}(y)
\int_{a}^{b}  K(x, y) f(x) \diamondsuit_{\alpha} x \right )
\psi(y) G^{\frac{1}{q}}(y) g(y) \diamondsuit_{\alpha} y \\
& \leq \left( \int_{a}^{b}  G^{1-p}(y) \psi^{-p}(y) \left (
\int_{a}^{b} K(x, y) f(x) \diamondsuit_{\alpha} x \right
)^{p}\diamondsuit_{\alpha} y \right )^{\frac{1}{p}} \\
& \qquad \qquad \cdot \left (
\int_{a}^{b} \psi^{q}(y) G(y) g^{q}(y) \diamondsuit_{\alpha} y
\right)^{\frac{1}{q}} \, .
\end{split}
\end{equation*}
Using \eqref{eq6}, we get that
\begin{multline*}
\int_{a}^{b} \int_{a}^{b} K(x, y) f(x) g(y) \diamondsuit_{\alpha}
x \diamondsuit_{\alpha} y  \\
\leq \left ( \int_{a}^{b}
\varphi^{p}(x) F(x) f^{p}(x)\diamondsuit_{\alpha} x \right
)^{\frac{1}{q}}  \left ( \int_{a}^{b} \psi^{q}(y) G(y) g^{q}(y)
\diamondsuit_{\alpha} y \right)^{\frac{1}{q}} \, ,
\end{multline*}
which completes the proof.
\end{proof}

\begin{remark}
Choose $\mathbb{T}= \mathbb{R}$. In this particular case the inequalities
\eqref{eq5} and \eqref{eq6} give the Hardy type inequalities
proved in \cite{krm}. If
\begin{equation}
\label{eq:rem:krm}
\left ( f(x)\frac{\varphi(x)}{\psi(y)}\right )^{p} = K\left (
g(y)\frac{\psi(y)}{\varphi(x)}\right )^{q},
\end{equation}
then \eqref{eq5} takes the form of equality.
In this case there exist arbitrary constants $A$ and $B$,
not both zero, such that
$$
f^{p}(x)= A \varphi^{-(p+q)}(x) \mbox{ and } g^{q}(y)= B
\psi^{-(p+q)}(y).
$$
This is possible only if
$$
\int_{a}^{b} F(x) \varphi^{-q}(x) \diamondsuit_{\alpha} x < \infty
\mbox{ and } \int_{a}^{b} G(y) \psi^{-p}(y) \diamondsuit_{\alpha}
y < \infty  \, .
$$
If \eqref{eq:rem:krm} does not hold,
inequalities in Theorem~\ref{thm:da:Hineq} are strict.
\end{remark}

As corollaries of Theorem~\ref{thm:da:Hineq}
we have the following results.

\begin{corollary}
\label{thm37} Let $\mathbb{T}$ be a time scale, $a$, $b \in
\mathbb{T}$ with $a < b$, $h(y)$, $f(x)$, $g(y)$, $\varphi(x)$,
and $\psi(y)$ be nonnegative functions, and
$\frac{1}{p}+\frac{1}{q}=1$ with $p>1$. Setting $H(y)= h(y)
\psi^{-p}(y)$, then the two inequalities
\begin{multline*}
 \int_{a}^{b} \int_{a}^{y} h(y) f(x)g(y) \diamondsuit_{\alpha} x
\diamondsuit_{\alpha} y \\
 \leq \left( \int_{a}^{b} \varphi^{p}(x) f^{p}(x) \left(
\int_{x}^{b} H(y) \diamondsuit_{\alpha} y \right )
\diamondsuit_{\alpha} x \right )^{\frac{1}{p}} \\
\qquad \qquad \left( \int_{a}^{b}
\psi^{q}(y) g^{q}(y) h(y) \left( \int_{a}^{y} \varphi^{-q}(x)
\diamondsuit_{\alpha} x \right ) \diamondsuit_{\alpha} y \right
)^{\frac{1}{q}}
\end{multline*}
and
\begin{multline*}
 \int_{a}^{b} H(y) \left ( \int_{a}^{y} \varphi^{-q}
 \diamondsuit_{\alpha}x \right)^{1-p} \left ( \int_{a}^{y} f(x)
 \diamondsuit_{\alpha}x\right)^{p}\diamondsuit_{\alpha}y \\
 \leq \left( \int_{a}^{b} \varphi^{p}(x) f^{p}(x)
 \left(
\int_{x}^{b} H(y) \diamondsuit_{\alpha} y \right )
\diamondsuit_{\alpha} x \right )^{\frac{1}{p}}
\end{multline*}
hold and are equivalent.
\end{corollary}
\begin{proof}
Use Theorem~\ref{thm:da:Hineq} with $
 K(x, y)=
 \left\{
 \begin{array}{rll}
 h(y), & \mbox{ if } & x \leq y \\
               0, & \mbox{ if } & x > y  \, .
\end{array}
\right. $
\end{proof}

\begin{corollary}
\label{thm38} Let $\mathbb{T}$ be a time scale, $a$, $b \in
\mathbb{T}$ with $a < b$, $h(y)$, $f(x)$, $g(y)$, $\varphi(x)$,
and $\psi(y)$ be nonnegative, and
$\frac{1}{p}+\frac{1}{q}=1$ with $p> 1$. Then, the two
inequalities
\begin{multline*}
 \int_{a}^{b} \int_{y}^{b} h(y) f(x)g(y) \diamondsuit_{\alpha} x
\diamondsuit_{\alpha} y \\
 \leq \left( \int_{a}^{b} \varphi^{p}(x) f^{p}(x) \left(
\int_{a}^{x} H(y) \diamondsuit_{\alpha} y \right )
\diamondsuit_{\alpha} x \right )^{\frac{1}{p}} \\
\left( \int_{a}^{b}
\psi^{q}(y) g^{q}(y) h(y) \left( \int_{y}^{b} \varphi^{-q}(x)
\diamondsuit_{\alpha} x \right ) \diamondsuit_{\alpha} y \right
)^{\frac{1}{q}},
\end{multline*}
and
\begin{multline*}
\int_{a}^{b} H(y) \left ( \int_{y}^{b} \varphi^{-q}
 \diamondsuit_{\alpha}x \right)^{1-p} \left ( \int_{y}^{b} f(x)
 \diamondsuit_{\alpha}x\right)^{p}\diamondsuit_{\alpha}y \\
\leq \left( \int_{a}^{b} \varphi^{p}(x) f^{p}(x) \left(
\int_{a}^{x} H(y) \diamondsuit_{\alpha} y \right )
\diamondsuit_{\alpha} x \right )^{\frac{1}{p}}
\end{multline*}
hold and are equivalent.
\end{corollary}
\begin{proof}
Use Theorem~\ref{thm:da:Hineq} with $
 K(x, y)=
 \left\{
 \begin{array}{rll}
 0, & \mbox{ if } & x \leq y \\
               h(y), & \mbox{ if } & x > y  \, .
\end{array}
\right. $
\end{proof}

\begin{remark}
When $\alpha = 1$,  Corollaries~\ref{thm37} and \ref{thm38}
coincide, respectively, with Theorems~7 and 8 in \cite{adnan}.
In the particular case $\mathbb{T}=\mathbb{R}$, they give
Theorem~3 and Theorem~4 of \cite{krm}.
\end{remark}

It is interesting to consider the case when functions $F(x)$
and $G(y)$ of Theorem~\ref{thm:da:Hineq} are bounded.  We then
obtain the following:
\begin{theorem}
Let $\frac{1}{p}+ \frac{1}{q}=1$ with $p>1, K(x, y), f(x), g(y),
\varphi(x), \psi(x)$ be nonnegative functions and $F(x) =
\int_{a}^{b} \frac{K(x, y)}{\psi^{p}(y)} \diamondsuit_{\alpha} y
\leq F_{1}(x) \, , G(y) = \int_{a}^{b} \frac{K(x,
y)}{\varphi^{q}(x)} \diamondsuit_{\alpha} x \leq G_{1}(y) $. Then, the inequalities
\begin{multline}\label{eq7}
\int_{a}^{b} \int_{a}^{b} K(x, y) f(x)g(y) \diamondsuit_{\alpha} x
\diamondsuit_{\alpha} y \\
 \leq \left (\int_{a}^{b} \varphi^{p}(x) F_{1}(x) f^{p}(x)
\diamondsuit_{\alpha} x \right)^{\frac{1}{p}}   \left (
\int_{a}^{b} \psi^{q}(y) G_{1}(y) g^{q}(y) \diamondsuit_{\alpha} y \right)^{\frac{1}{p}}
\end{multline}
and
\begin{equation}\label{eq8}
\int_{a}^{b} G_{1}^{1-p}(y) \psi^{-p}(y) \left ( \int_{a}^{b} K(x,
y) f(x) \diamondsuit_{\alpha} x \right)^{p} \diamondsuit_{\alpha}
y \leq \int_{a}^{b} \varphi^{p}(x) F_{1}(x)
f^{p}(x)\diamondsuit_{\alpha} x
\end{equation}
hold and are equivalent.
\end{theorem}

The following result extends
the one found in \cite{sulaiman}.

\begin{theorem}
\label{lst:thm}
Let $F, G, L(f, g), M(f)$, and $N(g)$ be positive functions,
$p> 1$, $\frac{1}{p}+ \frac{1}{q}=1$, such that
$$
0 < \int_{a}^{b} M^{p}(f(t)) F^{p}(t) \diamondsuit_{\alpha} t <
\infty \, , \quad
0 < \int_{c}^{d} N^{q}(g(t)) G^{q}(t)
\diamondsuit_{\alpha} t < \infty \, .
$$
Then, the inequalities
\begin{multline}\label{eq9}
\int_{a}^{b}  \int_{c}^{d} \frac{F(x)G(y)}{L(f(x), g(y))}
\diamondsuit_{\alpha}x \diamondsuit_{\alpha} y \\
\leq C \left (\int_{a}^{b} M^{p}(f(t))
F^{p}(t)\diamondsuit_{\alpha} t \right )^{\frac{1}{p}} \left
(\int_{c}^{d} N^{q}(g(t)) G^{q}(t)\diamondsuit_{\alpha} t \right
)^{\frac{1}{q}}
\end{multline}
and
\begin{equation}\label{eq10}
\int_{c}^{d} N^{-p}(g(y)) \left ( \int_{a}^{b} \frac{F(x)}{L(f(x),
g(y))} \diamondsuit_{\alpha} x \right )^{p} \diamondsuit_{\alpha}
y  \leq C^{p}\int_{a}^{b} M^{p}(f(t)) F^{p}(t)
\diamondsuit_{\alpha} t \, ,
\end{equation}
where $C$ is a constant, are equivalent.
\end{theorem}

\begin{proof}
Suppose that the inequality \eqref{eq10} is valid. Then,
\begin{equation*}
\begin{split}
& \int_{a}^{b} \int_{c}^{d} \frac{F(x)G(y)}{L(f(x), g(y))}
\diamondsuit_{\alpha}x \diamondsuit_{\alpha} y \\
& = \int_{c}^{d} N(g(y)) G(y) \left ( N^{-1}(g(y)) \int_{a}^{b}
\frac{F(x)}{L(f(x), g(y))} \diamondsuit_{\alpha}x
\right)\diamondsuit_{\alpha}y \\
&\leq  \left ( \int_{c}^{d} N^{q}(g(y)) G^{q}(y)
\diamondsuit_{\alpha} y \right )^{\frac{1}{q}} \left (
\int_{c}^{d} N^{-p}(g(y)) \left ( \int_{a}^{b} \frac{F(x)}{L(f(x),
g(y))} \diamondsuit_{\alpha}x
\right )^{p}\diamondsuit_{\alpha}y  \right)^{\frac{1}{p}}\\
& \leq C^{p}\left ( \int_{a}^{b} M^{p}(f(t)) F^{p}(t)
\diamondsuit_{\alpha} t \right )^{\frac{1}{p}} \left (
\int_{c}^{d} N^{q}(g(t)) G^{q}(t) \diamondsuit_{\alpha} t \right
)^{\frac{1}{q}}.
\end{split}
\end{equation*}
We just proved inequality \eqref{eq9}.
Let us now suppose that the inequality \eqref{eq9} is valid.
By setting $G(y)= N^{-p}(g(y)) \left (\int_{a}^{b}
\frac{F(x)}{L(f(x), g(y))} \diamondsuit_{\alpha}x \right
)^{\frac{p}{q}} \diamondsuit_{\alpha} y $ and applying \eqref{eq9}, we obtain that
\begin{equation*}
\begin{split}
& \int_{c}^{d} N^{-p}(g(y)) \left ( \int_{a}^{b}
\frac{F(x)}{L(f(x), g(y))} \diamondsuit_{\alpha} x \right )^{p}
\diamondsuit_{\alpha} y  \\
& = \int_{c}^{d} \left ( \int_{a}^{b} \frac{F(x)}{L(f(x), g(y))}
\diamondsuit_{\alpha}x \right ) N^{-p}(g(y)) \left (\int_{a}^{b}
\frac{F(x)}{L(f(x), g(y))} \diamondsuit_{\alpha}x \right
)^{\frac{p}{q}} \diamondsuit_{\alpha} y \\
& \leq C \left ( \int_{a}^{b} M^{p}(f(x)) F^{p}(x)
\diamondsuit_{\alpha} x \right )^{\frac{1}{p}}\\
& \qquad \times \left ( \int_{c}^{d}  N^{q}(g(y)) N^{-pq}(g(y)) \left
(\int_{a}^{b} \frac{F(x)}{L(f(x), g(y))} \diamondsuit_{\alpha}x
\right
)^{p}\diamondsuit_{\alpha}y \right )^{\frac{1}{q}}\\
& =  C \left ( \int_{a}^{b} M^{p}(f(x)) F^{p}(x)
\diamondsuit_{\alpha} x \right )^{\frac{1}{p}} \\
& \qquad \qquad \left (
\int_{c}^{d} N^{-p}(g(y)) \left (  \int_{a}^{b}
\frac{F(x)}{L(f(x), g(y))} \diamondsuit_{\alpha}x \right)^{p}
\diamondsuit_{\alpha}y \right )^{\frac{1}{q}}.
\end{split}
\end{equation*}
It follows \eqref{eq10}:
\begin{equation*}
\int_{c}^{d} N^{-p}(g(y)) \left ( \int_{a}^{b} \frac{F(x)}{L(f(x),
g(y))} \diamondsuit_{\alpha} x \right )^{p} \diamondsuit_{\alpha}
y  \leq C^{p} \int_{a}^{b} M^{p}(f(t)) F^{p}(t)
\diamondsuit_{\alpha} t.
\end{equation*}
\end{proof}


\section{Conclusion}

The study of integral inequalities on time scales via the
diamond-$\alpha$ integral, which is defined as a linear
combination of the delta and nabla integrals, plays an important
role in the development of the theory of time scales
\cite{ozkan,Stef,srd,sidel}. In this paper we generalize some
delta-integral inequalities on time scales to diamond-$\alpha$
integrals. As special cases, one obtains previous H\"{o}lder's
and Hardy's inequalities.
\\
\\
{\bf Acknowledgements:}

Work partially supported by the {\it Centre for Research on
Optimization and Control} (CEOC) from the {\it Portuguese
Foundation for Science and Technology} (FCT), cofinanced by the
European Community Fund FEDER/POCI 2010.


\bibliographystyle{amsplain}


\end{document}